\documentclass[12pt,reqno]{amsart}
\usepackage{amsmath,amsfonts,amssymb,amsthm,amscd,latexsym}

\usepackage{color}

\newtheorem{theorem}{Theorem}[section]
\newtheorem{lemma}[theorem]{Lemma}
\newtheorem{lem}[theorem]{Lemma}

\newtheorem{corollary}[theorem]{Corollary}

\newtheorem{remark}[theorem]{Remark}

\newtheorem{ques}[theorem]{Question}

\newcommand{\cM}{{\mathcal M}}
\newcommand{\cN}{{\mathcal N}}
\newcommand{\cA}{{\mathcal A}}
\newcommand{\cB}{{\mathcal B}}

\newcommand{\cR}{{\mathcal R}}

\begin{document}


\title[Ring isomorphisms of Murray--von Neumann algebras]{Ring isomorphisms
of Murray--von Neumann algebras}

\author[Sh. A. Ayupov]{Shavkat Ayupov}
\address{V.I.Romanovskiy Institute of Mathematics\\
  Uzbekistan Academy of Sciences\\ 81,  Mirzo Ulughbek street, 100170  \\
  Tashkent,   Uzbekistan}
\address{National University of Uzbekistan \\
4, University street, 100174, Tashkent, Uzbekistan}
\email{\textcolor[rgb]{0.00,0.00,0.84}{shavkat.ayupov@mathinst.uz}}

\author[K. Kudaybergenov]{Karimbergen Kudaybergenov}
\address{V.I.Romanovskiy Institute of Mathematics\\
  Uzbekistan Academy of Sciences \\ 81, Mirzo Ulughbek street, 100170  \\
  Tashkent,   Uzbekistan}
  \address{Department of Mathematics\\
 Karakalpak State University\\
 1, Ch. Abdirov, 230112,  Nukus, Uzbekistan}
\email{\textcolor[rgb]{0.00,0.00,0.84}{karim2006@mail.ru}}

\newcommand{\M}{\mathcal{M}}
\newcommand{\sm}{S(\mathcal{M})}

\begin{abstract}
We give a complete description of ring isomorphisms between algebras  of measurable operators affiliated with  von Neumann algebras of type II$_1.$

\end{abstract}


\maketitle

\bigskip

\section{Introduction}

Let $\cM$ be a von Neumann algebra and let $S(\cM)$ (respectively, $LS(\cM)$) be a $\ast$-algebra of all measurable (respectively, locally measurable) operators with respect to $\cM.$

In the paper \cite{MMori2020} M. Mori  characterized lattice isomorphisms between projection lattices  $P(\cM)$
and $P(\cN)$  of arbitrary von Neumann algebras $\cM$   and $\cN$, respectively,  by means of ring isomorphisms between the algebras $LS(\cM)$ and $LS(\cN)$.  In this connection he investigated the following problem.

\begin{ques}\label{ques}
Let $\cM, \cN$ be von Neumann algebras. What is the general form of ring
isomorphisms from $LS(\cM)$ onto $LS(\cN)?$
\end{ques}

In \cite[Theorem B]{MMori2020} Mori himself gave  an answer to the above Question in the case of von Neumann algebras of type I$_\infty$ and III. Namely,
any ring isomorphism $\Phi$ from  $LS(\cM)$ onto $LS(\cN)$ has the form
$$
\Phi(x)=y\Psi(x)y^{-1},\, x\in LS(\cM),
$$
where $\Psi$ is a real $\ast$-isomorphism from $LS(\cM)$ onto $LS(\cN)$ and
$y\in LS(\cN)$ is an invertible element.
Note that in the case  where $\Phi$ is an  algebraic isomorphism  of type I$_\infty$  von Neumann algebras, the above
presentation was obtained in \cite{AAKD11}.

If  $\cM$ is a finite von
Neumann algebra, then $LS(\cM) = S(\cM)$ (see \cite{MC}).
If the von Neumann algebra $\cM$ is abelian (i.e. of type I$_1$) then it is $\ast$-isomorphic to the algebra
   $L_{\infty}(\Omega, \Sigma,
\mu)$ of all (classes of equivalence of)  essentially bounded measurable complex functions on a measure space
$(\Omega, \Sigma, \mu)$ and therefore,  $S(\cM)\cong
S(\Omega, \Sigma,
\mu)$ is the algebra of all   measurable complex functions on
$(\Omega, \Sigma, \mu).$
A.G. Kusraev
\cite{Kus} by means of Boolean-valued analysis establishes
necessary and sufficient conditions for existence of discontinuous non
trivial  algebra automorphisms on extended complete complex
$f$-algebras. In particular, he has proved that the algebra $S[0, 1]$ (which is isomorphic to  $LS(L_\infty[0,1])=S(L_\infty[0,1])$) admits
discontinuous  algebra automorphisms which identically act on the Boolean algebra  $P(L_\infty[0,1])$ of characteristic functions of measurable
subsets of the interval $[0,1].$

The following consideration shows that also  for the type I$_n$   case, $1<n<\infty,$ ring isomorphisms may be discontinuous in general (see for details  \cite{AAKD11})
 and therefore the representation from \cite[Theorem B]{MMori2020} is not valid for this case.

Let $\cM$ be a von Neumann algebra of type I$_n,$  $1<n<\infty$  with the center $Z(\cM).$
Then $\cM$ is $\ast$-isomorphic to the algebra $M_n(Z(\cM))$ of all $n\times n$ matrices over $Z(\cM)$
(cf. \cite[Theorem 2.3.3]{Sakai_book}). Moreover the algebra $S(\cM)$ is $\ast$-isomorphic
to the algebra $M_n(Z(S(\cM))),$ where $Z(S(\cM)) = S(Z(\cM))$ is the center of $S(\cM)$
(see \cite[Proposition 1.5]{Alb2}).
For an arbitrary  von Neumann algebra  $\cM$ of type I$_n$ each  algebra automorphism $\Phi$ of $S(\cM)$ can be  represented in the form
$$
\Phi(x) = a\overline{\Psi}(x)a^{-1},\, x\in S(\cM),
$$
where $a\in S(\cM)$  is an invertible element and
$\overline{\Psi}$ is an extension of a $\ast$-automorphism $\Psi$ of the
center $S(Z(\cM)).$

In \cite{MMori2020} the author  conjectured that the  representation of ring isomorphisms, mentioned above for type  I$_\infty$ and III cases
 holds also for type II von Neumann algebras. At the end of the paper M.~Mori wrote that
''
The author does not know whether
or not such a $\Phi$  is automatically real-linear even in the case $\cM$ and $\cN$ are (say,
approximately finite dimensional) II$_1$ factors. Note that $LS(\cM)$ cannot have a
Banach algebra structure because of the fact that an element of $LS(\cM)$ can have an
empty or dense spectral set. Hence it seems to be difficult to make use of {\it automatic
continuity} results on algebra isomorphisms as in \cite{Dales}''.

In the present paper we give an answer to the Question~\ref{ques}  for  type II$_1$ von Neumann algebras.
The paper is  organize as follows.

In Section 2 we give definitions of various kinds of isomorphisms between $\ast$-algebras and also some preliminaries from the theory
 of measurable operators affiliated with von Neumann algebras.

In order to prove the main result of the present paper, in  Sections 3 and 4 we  show  automatic real-linearity and automatic continuity of ring isomorphisms between algebras of measurable operators affiliated with von Neumann algebras of  type II$_1.$  Namely, we prove the following two
theorems.

\begin{theorem}\label{ringisomorphism}
Let $\mathcal{M}$ and $\cN$ be  type II$_1$ von Neumann algebras.
Then any ring isomorphism from  $S(\mathcal{M})$ onto $S(\mathcal{N})$ is a real algebra isomorphism.
\end{theorem}

\begin{theorem}\label{realisomorphism}
Let $\mathcal{M}$ and $\cN$ be type II$_1$ von Neumann algebras.
Then any ring isomorphism from  $S(\mathcal{M})$ onto $S(\mathcal{N})$ is continuous in the  local measure topology.
\end{theorem}

In Section 5 the following main result  confirms the  Conjecture 5.1 in \cite{MMori2020} and answers  the above Question 1.1 for the type II$_1$ case.
\begin{theorem}\label{latticeisomorphism}
Let $\cM$ and $\cN$ be von Neumann algebras of type II$_1.$ Suppose that
$\Phi: S(\cM) \to  S(\cN)$ is a ring isomorphism.
Then there exist an invertible element
$a \in S(\cN)$ and a real
$\ast$-isomorphism $\Psi: \cM \to  \cN$ (which extends to a real
$\ast$-isomorphism from $S(\cM)$ onto $S(\cN)$) such that
$\Phi(x) = a\Psi(x)a^{-1}$ for
all $x \in  S(\cM).$
\end{theorem}

\begin{corollary}\label{latticering}
Let $\cM$ and $\cN$  be von Neumann algebras of type II$_1.$ The projection lattices $P(\cM)$
and $P(\cN)$ are lattice isomorphic, if and only if the von Neumann algebras $\cM$ and $\cN$ are real
$\ast$-isomorphic  (or equivalently, $\cM$ and $\cN$ are
Jordan
$\ast$-isomorphic).
\end{corollary}

\section{Preliminaries}

\subsection{Various isomorphisms of $\ast$-algebras}
For
$\ast$-algebras $\cA$  and
$\cB,$ a (not necessarily linear) bijection $\Phi: \cA \to  \cB$ is called
\begin{itemize}
\item a ring isomorphism if it is additive and multiplicative;
\item a real algebra isomorphism if it is a real-linear ring isomorphism;
\item an algebra isomorphism if it is a complex-linear ring isomorphism;
\item a real
$\ast$-isomorphism if it is a real algebra isomorphism and satisfies $\Phi(x^\ast) =
\Phi(x)^\ast$ for all $x \in  \cA;$
\item a
$\ast$-isomorphism if it is a complex-linear real $\ast$-isomorphism. 
\end{itemize}

\subsection{von Neumann algebras}

Let $H$  be a Hilbert space,  $B(H)$ be the $\ast$-algebra of all bounded linear operators
on $H$ and let $\M$ be a von Neumann algebra in $B(H)$.

Denote by $P(\mathcal{M})$ the set of all projections in $\mathcal{M}.$ Recall that two projections $e, f \in  P(\mathcal{M})$ are called \textit{equivalent}  (denoted as $e\sim f$) if there exists an element
$u \in \mathcal{M}$ such that $u^\ast  u = e$ and $u u^\ast  = f.$
For projections $e, f \in  \mathcal{M}$
notation $e \precsim  f$ means that there exists a projection $q \in  \mathcal{M}$ such that
$e\sim q \leq f.$ A projection $p \in \mathcal{M}$ is said to be \textit{finite}, if it is not equivalent to its proper sub-projection, i.e.
the conditions $q \leq  p$ and $q\sim p$ imply that $q = p.$

It is known \cite[Theorem 8.4.3]{KRII} that for a  finite von Neumann algebra $\mathcal{M}$ with the center $Z(\mathcal{M})$  and
the set  $P(\mathcal{M})$ of all projections  in $\mathcal{M},$    there exists a unique mapping $\Delta: P(\mathcal{M})\to Z(\mathcal{M})$
such that
\begin{itemize}
\item[(i)]  $0\neq \Delta(e)\ge 0$ if $e\in P(\mathcal{M})$ and $e\neq 0;$
\item[(ii)] $\Delta(e+f)=\Delta(e)+\Delta(f)$ if $e,f\in P(\mathcal{M})$ and $ef=0;$
\item[(iii)] $\Delta(e)=\Delta(f)$ if and only if $e\sim f;$
\item[(iv)] $\Delta(e)=e$ if $e\in P(\mathcal{M})\cap Z(\mathcal{M}).$
\end{itemize}
The mapping $\Delta$ is called the dimension function.

The following is a well-known result which is crucial in our further constructions, and
 it asserts that in an arbitrary type II$_1$ von Neumann algebra there exists a copy of the hyperfinite type II$_1$ factor $\cR \cong \otimes _{k=1}^
 \infty M_2 (\mathbb{C})$.

\begin{lemma}\label{lem_matrix_units}
Let $\cN$ be a type II$_1$ von Neumann algebra.
There is  a  system of matrix units $\left\{e^{(n)}_{ij}:\  n=0,1,\dots, \ i,j=1,\dots,2^n\right\}$ in  $\cN$ (here $e^{(0)}_{1,1}=\mathbf{1}$) such that
\begin{itemize}
\item[(a)] $e^{(n)}_{ij}e^{(n)}_{kl}=\delta_{jk}e^{(n)}_{il},$ where $\delta_{jk}$ is the Kronecker delta;
\item[(b)] $\left(e^{(n)}_{ij}\right)^\ast=e^{(n)}_{ji};$
\item[(c)] $e^{(n-1)}_{ij}=e^{(n)}_{2i-1,2j-1}+e^{(n)}_{2i,2j}$
\end{itemize}
for all $n=1,2,\ldots.$
\end{lemma}

\subsection{Murray-von Neumann algebra}

A densely defined closed linear operator $x : \textrm{dom}(x) \to  H$
(here the domain $\textrm{dom}(x)$ of $x$ is a dense linear subspace in $H$) is said to be \textit{affiliated} with $\mathcal{M}$
if $yx \subset  xy$ for all $y$ from the commutant $\mathcal{M}'$  of the algebra $\mathcal{M}.$

A linear operator $x$ affiliated with $\mathcal{M}$ is called \textit{measurable} with respect to $\mathcal{M}$ if
$e_{(\lambda,\infty)}(|x|)$ is a finite projection for some $\lambda>0.$ Here
$e_{(\lambda,\infty)}(|x|)$ is the  spectral projection of $|x|$ corresponding to the interval $(\lambda, +\infty).$
We denote the set of all measurable operators by $S(\mathcal{M}).$

Let $x, y \in  S(\mathcal{M}).$ It is well known that $x+y$ and
$xy$ are densely-defined and preclosed
operators. Moreover, the (closures of) operators $x + y, xy$ and $x^\ast$  are also in $S(\mathcal{M}).$
When
equipped with these operations, $S(\mathcal{M})$ becomes a unital $\ast$-algebra over $\mathbb{C}$  (see \cite{Segal}). It
is clear that $\mathcal{M}$  is a $\ast$-subalgebra of $S(\mathcal{M}).$
In the case of finite von Neumann algebra $\M$, all operators affiliated with $\M$ are measurable and the algebra $\sm$ is referred to as the \emph{Murray-von Neumann algebra} associated with $\M$ (see \cite{KL}).

Let $\tau$ be a faithful normal finite trace on $\mathcal{M}.$
Consider the topology  $t_\tau$ of convergence in measure or \textit{measure topology} \cite{Nel}
on $S(\mathcal{M}),$ which is defined by
the following neighborhoods of zero:
$$
N(\varepsilon, \delta)=\{x\in S(\mathcal{M}): \exists \, e\in P(\mathcal{M}), \, \tau(\mathbf{1}-e)\leq\delta, \, xe\in
\mathcal{M}, \, \|xe\|_\cM\leq\varepsilon\},
$$
where $\varepsilon, \delta$
are positive numbers. The pair $(\sm, t_\tau)$ is a complete topological $\ast$-algebra.

 Recall, that an operator $x$ affiliated with $\mathcal{M}$ is called locally measurable (with respect to $\mathcal{M}$)
if there is a sequence $\{z_n\}_{n=0}^\infty\subset
Z(\mathcal{M})$  such that $z_n \uparrow \mathbf{1}$ and such that
$z_n (H)\subset  {\rm dom}(x)$ and $xz_n \in  S(\mathcal{M})$ for every $n \geq 0.$
We denote by $LS(\mathcal{M})$ the $\ast$-algebra of all locally measurable operators, with respect to the operations of strong sum and strong product.

Let $\cM$ be a finite von Neumann algebra with a faithful normal semi-finite trace $\tau.$ Then there exists a
family $\{z_i\}_{i\in I}$ of mutually orthogonal central projections in $\cM$ with $\bigvee\limits_{i\in I} z_i=\mathbf{1}$ and
such that $\tau(z_i)< +\infty$
for every $i\in  I$ (such a family exists because $\cM$ is a finite algebra). Then the algebra $S(\cM)$ is $\ast$-isomorphic
to the algebra $\prod_{i\in I}S(z_i\cM)$ (with the coordinate-wise operations and involution), i.e.
$$
S(\cM) \cong \prod_{i\in I}S(z_i\cM),
$$
($\cong$ denoting $\ast$-isomorphism of algebras) \cite{MC}.
This property implies that given any family $\{z_i\}_{i\in I}$ of mutually orthogonal central projections in $\cM$
with $\bigvee\limits_{i\in I} z_i=\mathbf{1}$ and a family of elements $\{x_i\}_{i\in I}$ in $S(\cM)$ there exists a unique element $x \in S(\cM)$ such
that $z_i x = z_ix_i$ for all $i \in  I.$
Let $t_{\tau_i}$
be the measure topology on $S(z_i\cM) = S(z_i\cM, \tau_i),$ where $\tau_i = \tau|_{z_i\cM},$ $i\in I.$ On the algebra
$S(\cM) \cong\prod\limits_{i\in I}S(z_i\cM)$ we consider the topology $t$ which is the Tychonoff product of the topologies  $t_{\tau_i},$
$i\in I.$ This topology coincides with so-called \textit{local measure topology} on $S(\cM)$ (see \cite[Remark 2.7]{MC}).

Let $\cM$ be a finite von Neumann algebra. A $\ast$-subalgebra  $\mathcal{A}$ of $S(\mathcal{M})$ is said to be
\emph{regular}, if it is a regular ring in the sense of von Neumann, i.e., if for every
$a\in\mathcal{A}$ there exists an element  $b\in\mathcal{A}$ such that
$aba=a.$

Given $a\in S(\cM)$  let  $a=v|a|$ be the polar decomposition of $a.$
Then $l(a) =v v^\ast$ and  $r(a)=v^\ast v$ are left and right supports of the element  $a$, respectively.
The projection  $s(a)=l(a)\vee r(a)$ is the support of the element $a$. It is clear that $r(a)=s(|a|)$ and  $l(a)=s(|a^*|)$.
There is a unique element $i(a)$ in
$S(\mathcal{M})$ such that $ai(a)=l(a),\ i(a)a=r(a),\ ai(a)a=a,$
$i(a)l(a)=i(a)$ and  $r(a)i(a)=i(a).$
The element  $i(a)$ is called the \emph{partial inverse} of the element $a.$
Therefore  $S(\mathcal{M})$ is a regular $*$-algebra  (see \cite{Berber}, \cite{Saito}).

    Let $e\in S(\cM)$ be an idempotent, i.e., $e^2=e.$ Then
    \begin{equation}\label{rangeofidem}
        l(e)e=e,\,\, el(e)=l(e).
    \end{equation}
    Indeed, the first equality is the definition of the left projection. Using equality $ei(e)=l(e)$ we obtain that
    \begin{align*}
        el(e)=e(ei(e))=e^2i(e)=ei(e)=l(e).
    \end{align*}
    Note that in \cite[Theorem 1.3]{Koliha} the existence of range projections with the above two properties is proved for bounded operators.

    Let us consider the decomposition $e=l(e)+u,$ where $u=e-l(e).$ Then $u^2=0.$ Indeed,
    \begin{align*}
        u^2 &= (e-l(e))^2=e^2-el(e)-l(e)e+l(e)^2=e-l(e)-e+l(e)=0.
    \end{align*}

    Later in Section 5 we need the following Lemma.
    \begin{lemma}\label{limitofrang}
        Let $\{e_n\}\subset S(\cM)$ be a sequence of idempotents such that $e_n\to e\in P(\cM)$ in the measure topology. Then $l(e_n)\to e$ in the same topology.
    \end{lemma}

    \begin{proof}
        Let us consider the decomposition $e_n=l(e_n)+u_n,$ $n\in \mathbb{N}.$ Then from the continuity of the involution in the measure topology we have
        \begin{align*}
            u_n -u_n^\ast &= (e_n-l(e_n))-(e_n-l(e_n))^\ast = e_n-l(e_n)-e_n^\ast+l(e_n)=e_n-e_n^\ast\to e-e^\ast=0,
        \end{align*}
        because $e=e^\ast\in P(\cM).$
        Thus $(u_n-u_n^\ast)^2\to 0$ in the measure topology. Since
        $u_n^2=0=(u_n^\ast)^2,$ it follows that
        $u_n u_n^\ast+u_n^\ast u_n\to 0.$ Further, from $0\le u_n^\ast u_n\le u_n u_n^\ast+u_n^\ast u_n,$ we have that $|u_n|^2=u_n^\ast u_n\to 0,$ thus
        $u_n\to 0.$ So,
        $l(e_n)=e_n-u_n\to e$ in the measure topology.
    \end{proof}

\subsection{Symmetric $F$-norm  on $S(\cM)$}

For the convenience of the reader, we recall the definition of $F$-norms.
Let $E$ be a linear space over the field $\mathbb{C}$.
A function $\left\|\cdot\right\|$ from $E$ to $\mathbb{R}$ is an $F$-norm, if for all $x,y \in E$ the following properties hold:
\begin{align*}
\left\|x\right\| \geqslant 0 , ~\left\|x\right\| = 0 \Leftrightarrow x=0 ;\\
\left\|\alpha x\right\| \leqslant \left\|x\right\|, ~\forall~\alpha \in \mathbb{C},  |\alpha| \le 1 ;\\
\lim _{\alpha \rightarrow 0}\left\|\alpha x\right\| = 0;\\
\left\|x+y \right\| \le  \left\|x\right\|+\left\|y\right\| .
\end{align*}
The couple $(E, \left\|\cdot\right\|)$ is called an \emph{$F$-normed} space.
For more detailed information concerning $F$-normed spaces see \cite{KPR}.

Let  $\mathcal{M}$  be a type II$_1$ von Neumann algebra with a faithful normal finite  trace $\tau$
and let $\mathcal{E}$ be a linear subspace in $S(\cM)$
equipped with an $F$-norm~$\left\|\cdot\right\|_{\mathcal{E}}$.
We say that
$\mathcal{E}$ is a \textit{symmetrically $F$-normed  space}  if
for $x \in
\mathcal{E}$, $y\in S(\cM)$ the inequality $\mu(y)\leq \mu(x)$ implies that $y\in \mathcal{E}$ and
$\|y\|_\mathcal{E}\leq \|x\|_\mathcal{E}$.
Here, $\mu(x)$ stands for the singular value function of $x\in S(\cM)$ (see \cite{HS}).

\begin{remark}\label{2.3}The following function
$$
\left\|x\right\|_{S(\cM)} = \inf_{t>0}\left\{t+\mu(t;x)\right\},~x\in S(\cM) ,$$
is a symmetric $F$-norm $\left\|\cdot\right\|_{S(\cM)}$ on $S(\cM)$, that is,
$\left\|\cdot\right\|_{S(\cM)}$ is an $F$-norm on $S(\cM)$ \cite[Remark 3.4]{HS}.
Moreover, the topology induced by $\left\|\cdot\right\|_{S(\cM)}$ is equivalent to the measure topology \cite[Proposition 4.1]{HS}.
\end{remark}

\subsection{Reduction of the general    case to the case of a von Neumann algebra  with a faithful normal finite    trace}

Let  $\mathcal{M}$ and $\cN$ be arbitrary type II$_1$ von Neumann algebras with  faithful normal semi-finite traces $\tau_{\cM}$ and
$\tau_{\cN},$ respectively.

Consider  a ring isomorphism  $\Phi$ from $S(\cM)$ onto $S(\cN).$
There exists a family $\{z_i\}_{i\in I}$  of
mutually orthogonal central projections in $\mathcal{M}$ with
$\bigvee\limits_{i\in I} z_i=\mathbf{1}$  such that $\tau_\cM(z_i)<+\infty$ for every $i\in  I.$
Since any ring isomorphism maps $S(Z(\cM))$ onto $S(Z(\cN)),$  for each fixed $i\in I$ there exists a family $\{z_{i,j}\}_{j\in J}$  of
    mutually orthogonal central projections in $z_i\mathcal{M}$ with
    $\bigvee\limits_{j\in J} z_{i,j}=z_i$ such that $\tau_\cN(\Phi(z_{i,j}))<+\infty$ for every $j\in J.$
Since $\Phi$ sends  each central projection in $\cM$ to a central projection
in $\cN,$  $\Phi$ maps each $S(z_{i,j}\cM)$ onto
$S(\Phi(z_{i,j})\cN)\equiv \Phi(z_{i,j})S(\cN)$ for all $i\in I, j\in J.$  So, it suffices to consider
the type II$_1$ von Neumann algebras  $\mathcal{M}$ and $\cN$ with  faithful normal finite traces $\tau_{\cM}$ and
$\tau_{\cN},$ respectively.

In the next three Sections below  $\mathcal{M}$ and $\cN$ are supposed to be  arbitrary type II$_1$ von Neumann algebras with  faithful normal finite traces $\tau_{\cM}$ and
$\tau_{\cN},$ respectively.

\section{Real-linearity of ring isomorphisms}

In this section we prove Theorem \ref{ringisomorphism} in a series of Lemmas.

Suppose the contrary and assume that $\Phi$ is a  ring isomorphism from $S(\mathcal{M})$ onto
$S(\cN)$ which is not a real algebra isomorphism.

Let $\left\{e^{(n)}_{ij}:\  n=0,1,\dots, \ i,j=1,\dots,2^n\right\}$ be the system of matrix units as in Lemma~\ref{lem_matrix_units}.
Put $u_0=\mathbf{1}$  and for each $n\ge 1$ take the unitary
\begin{eqnarray}\label{unitary}
& u_n & = \sum\limits_{i=1}^{2^{n}-1}e_{i, i+1}^{(n)}+e_{2^n, 1}^{(n)}\in \cN.
\end{eqnarray}
In the proofs of real-linearity and automatic continuity of ring isomorphisms between the algebras $S(\cM)$ and $S(\cN)$ we  will essentially use this family of unitaries.

Denote
\begin{eqnarray*}\label{unitaryv}
v_n=\Phi^{-1}(u_n)\in S(\cM),\, n\ge 0.
\end{eqnarray*}
 Note that $u^{2^n}_n = \mathbf{1}\in \cN$  and thus $v^{2^n}_n = \Phi^{-1}\left(u^{2^n}_n\right)=\mathbf{1}\in \cM.$

\begin{lemma}\label{lemma3.4}
For $a \in  S(\mathcal{M})$ and a natural number $n\ge 0$ consider the following operator in $S(\cM)$
\begin{equation*}\label{defxn}
x  = \sum\limits_{i=0}^{2^{n}-1} (-1)^i v_n^{i} a   v_n^{-i}.
\end{equation*}
Then the following equalities hold:
        \begin{eqnarray}\label{vvvv}
    v_n x v_n^{-1}=-x
    \end{eqnarray}
    and
    \begin{eqnarray}\label{invv}
    v_m x v_m^{-1}=x
    \end{eqnarray}
    for all $m<n.$
\end{lemma}

\begin{proof} We have that
\begin{eqnarray*}
 &v_n x v_n^{-1} &= v_n\Big(\sum\limits_{i=0}^{2^{n}-1} (-1)^i v_n^{i} a    v_n^{-i}\Big)v_n^{-1}=
 -\sum\limits_{i=0}^{2^{n}-1} (-1)^{i+1} v_n^{i+1} a   v_n^{-(i+1)}=-x.
\end{eqnarray*}

We claim that
    \begin{eqnarray*}
    v_m & = & v_n^{2^{n-m}}
    \end{eqnarray*}
for all  $m<n.$
Since $v_n=\Phi^{-1}(u_n)$ and $\Phi$ is a ring isomorphism, it suffices to show that $u_{n+1}^2=u_n$ for the unitaries defined as in \eqref{unitary}. From the property (c)  of the matrix units we have
    \begin{eqnarray}\label{ccee}
        e_{i, i+1}^{(n)} & = &
        e_{2i-1, 2i+1}^{(n+1)} +e_{2i, 2i+2}^{(n+1)}.
    \end{eqnarray}
Using  the property  (a)   we obtain
    \begin{eqnarray*}
    &u_{n+1}^2 &    = \Big(\sum\limits_{i=1}^{2^{n+1}-1}e_{i, i+1}^{(n+1)}+e_{2^{n+1}, 1}^{(n+1)}\Big)
        \Big(\sum\limits_{j=1}^{2^{n+1}-1}e_{j, j+1}^{(n+1)}+e_{2^{n+1}, 1}^{(n+1)}\Big)\\
&=& \sum\limits_{i=1}^{2^{n+1}-2}e_{i, i+2}^{(n+1)}+e_{2^{n+1}, 2}^{(n+1)}+e_{2^{n+1}-1, 1}^{(n+1)}\\
&=& \sum\limits_{i=1}^{2^{n}-1}\left(e_{2i-1, 2i+1}^{(n+1)}+e_{2i, 2i+2}^{(n+1)}\right)+\left(e_{2^{n+1}, 2}^{(n+1)}+e_{2^{n+1}-1, 1}^{(n+1)}\right)\\
&\stackrel{\eqref{ccee}}{=}& \sum\limits_{i=1}^{2^{n}-1}e_{i, i+1}^{(n)}+ e_{2^{n}, 1}^{(n)}=u_n.
\end{eqnarray*}
Finally,
\begin{eqnarray*}
v_m x v_m^{-1}& = & v_n^{2^{n-m}} x v_n^{-2^{n-m}} =v_n^{2^{n-m}-1}\left(v_n x v_n^{-1}\right)v_n^{-2^{n-m}+1}\\
&\stackrel{\eqref{vvvv}}{=}& -v_n^{2^{n-m}-1}x v_n^{-2^{n-m}+1}\stackrel{\eqref{vvvv}}{=}\ldots= x.
\end{eqnarray*}
The proof of Lemma is complete.
\end{proof}

\begin{lemma}\label{rzkzk}
There exists $\varepsilon>0$ with the following property. Let $t_k$ be a rational number, $k\ge 0.$
Then there exists  a real number $\lambda_k$  such that
\begin{itemize}
    \item[(i)] $\left\|t_k x_k\right\|_{S(\cM)}\le \frac{1}{2^k};$
    \item[(ii)] $\displaystyle \left\|\Phi(x_k)\right\|_{S(\cN)}\ge \varepsilon,$
\end{itemize}
where
\begin{equation}\label{def_xk}
x_k  = \lambda_k\sum\limits_{i=0}^{2^{k}-1} (-1)^i v_k^{i} a_k   v_k^{-i}
\end{equation}
and  $a_k=\Phi^{-1}\left(e^{(k)}_{1,1}\right).$
\end{lemma}

\begin{proof}
 Since $\Phi$ is additive, it follows that  it is  rational linear, in particular, $\Phi(r\mathbf{1})=r\mathbf{1}$ for all $r\in \mathbb{Q}.$ By the assumption $\Phi$ is not $\mathbb{R}$-linear and therefore
the mapping
$$
\mu\in \mathbb{R}\to \Phi(\mu\mathbf{1})\in S(\cN)
$$
is discontinuous at the point $0.$
Therefore  there exists a number $\varepsilon>0$ and a sequence $\mu_1, \ldots, \mu_n, \ldots$ in $\mathbb{R}$ such that
\begin{center}
$\mu_n\to 0$ as $n\to\infty$ and $||\Phi(\mu_n\mathbf{1})||_{S(\cN)}\ge \varepsilon$ for all $n\in \mathbb{N}.$
\end{center}

Let us find the required  number $\lambda_k.$
Since  $\mu_n t_k \sum\limits_{i=0}^{2^{k}-1} (-1)^i v_k^{i} a_k
v_k^{-i}\to 0$ as $n\to \infty,$ there is a number $\mu_{n_k}$ such that
$\left\|t_k x_k\right\|_{S(\cM)}\le \frac{1}{2^k},$ where $x_k$ is the element of the form \eqref{def_xk} with  $\lambda_k=\mu_{n_k}.$

Further note that
\begin{eqnarray*}
\Phi(x_k) &=&\Phi\Big(\lambda_k\sum\limits_{i=0}^{2^{k}-1} (-1)^i v_k^{i} a_k   v_k^{-i}\Big) =
\Phi(\lambda_k\mathbf{1})\sum\limits_{i=0}^{2^{k}-1} (-1)^i u_k^{i} \Phi(a_n)   u_k^{-i}\\
&=&
\Phi(\lambda_k\mathbf{1}) \sum\limits_{i=0}^{2^{k}-1} (-1)^i u_k^{i} e_{1, 1}^{(k)}   u_k^{-i}=
\Phi(\lambda_k\mathbf{1}) \sum\limits_{i=1}^{2^{k}} (-1)^{i-1}e_{i, i}^{(k)}.
\end{eqnarray*}
Since $\Phi(\lambda_k\mathbf{1})$ is a central element in  $S(\cN),$ it follows that
$|\Phi(x_k)|=|\Phi(\lambda_k\mathbf{1})|.$
Hence $\left\|\Phi(x_k)\right\|_{S(\cN)}=\left\|\Phi(\lambda_k\mathbf{1})\right\|_{S(\cN)}\ge \varepsilon$
as required. The proof is complete.
\end{proof}

\begin{lemma}\label{constructiontk_x_k}
There exist  a sequence $\left\{x_k\right\}_{k\ge0}$ in $S(\cM)$ and an increasing sequence $\{t_k\}_{k\ge 0}$
  in $\mathbb{Q}$ with $t_k \to \infty$ such that
    \begin{itemize}
        \item[(1)] each $x_k$ is defined as in \eqref{def_xk};
        \item[(2)] $\displaystyle \left\|t_k x_k\right\|_{S(\cM)} \le  \frac{1}{2^k};$
        \item[(3)] $\displaystyle \left\|\Phi(x_k)\right\|_{S(\cN)}\ge \varepsilon;$
         \item[(4)] $\displaystyle \left\|\frac{1}{2t_k}\sum\limits_{i=0}^{k-1}u_k\Phi(t_i x_i)u_k^{-1}-\Phi(t_i x_i)\right\|_{S(\cN)} \le  \frac{\varepsilon}{2}.$
    \end{itemize}
    \end{lemma}

\begin{proof}
The proof is by induction on $k.$ For $k=0,$ since $u_0=\mathbf{1}$ and $t_0=1,$ it follows directly from Lemma~\ref{rzkzk}.

Now assume that we have constructed $x_0, \ldots, x_k$ and $t_0, \ldots, t_k$ with properties (1)-(4).

Due to the definition of $F$-norms, we have
$$
\lim_{t \rightarrow \infty}\left\|\frac{1}{2t}\sum\limits_{i=0}^{k}u_{k+1}\Phi(t_i x_i)u_{k+1}^{-1}-\Phi(t_i x_i)\right\|_{S(\cN)}= 0.
$$
Hence,
we can take $t_{k+1}>t_k+k$ such that
$$
\left\|\frac{1}{2t_{k+1}}\sum\limits_{i=0}^{k}u_{k+1}\Phi(t_ix_i)u_{k+1}^{-1}-\Phi(t_ix_i)\right\|_{S(\cN)} \le  \frac{\varepsilon}{2}.
$$

Further,  appealing to   Lemma~\ref{rzkzk}, we obtain   $x_{k+1}$ as in \eqref{def_xk} such that
$$
\displaystyle \left\|t_{k+1}x_{k+1}\right\|_{S(\cM)} \le  \frac{1}{2^{k+1}}
$$
and
$$
\displaystyle \left\|\Phi(x_{k+1})\right\|_{S(\cN)}\ge \varepsilon.
$$
Now, the sequences $\{x_k\}_{k=0}^\infty$ and $\{t_k\}_{k=0}^\infty $
satisfy the properties claimed in the lemma.
\end{proof}

\begin{proof}[Proof of Theorem~\ref{ringisomorphism}]

\

Let $x_0,\ldots, x_k,\ldots$ and $t_0,\ldots, t_k,\ldots$  be as in Lemma~\ref{constructiontk_x_k}. Set
$$
x=\sum\limits_{k=1}^\infty y_k,
$$
where $y_k=t_kx_k.$ Lemma~\ref{constructiontk_x_k} (2) guarantees that the above series is $||\cdot||_{S(\cM)}$-norm convergent in $S(\cM).$

In the proof below, we use the equality $v_n x_k v_n^{-1}=x_k$ for $k>n$ from Lemma \ref{lemma3.4}.
 Taking into account that $S(\cM),$  equipped with the measure topology, is  a topological algebra, in particular, the multiplication is continuous we obtain that
\begin{align*}
v_n x v_n^{-1}-x & =\sum\limits_{k=n+1}^\infty \left(v_n y_k v_n^{-1} - y_k\right)+\sum\limits_{k=1}^n \left(v_n y_k v_n^{-1} - y_k\right)\\
\stackrel{\eqref{invv}}{=}& \sum\limits_{k=1}^{n}  \left(v_n y_k v_n^{-1} - y_k\right)
\stackrel{\eqref{vvvv}}{=}
-2y_n + \sum\limits_{k=1}^{n-1}  \left(v_n y_k  v_n^{-1}- y_k\right)
\end{align*}
and therefore
\begin{align*}
u_n \Phi(x) u_n^{-1}  -  \Phi (x)  = &  \Phi\left(v_n x v_n^{-1}-x\right)= -2\Phi(y_n)
 +
\sum\limits_{k=1}^{n-1}  \left(u_n \Phi(y_k) u_n^{-1}-
\Phi(y_k)\right).
\end{align*}
Recalling that $y_k=t_kx_k$, we have
\begin{align*}
\frac{1}{2t_n}\left(u_n \Phi(x) u_n^{-1}  -  \Phi(x)\right) & =
-\Phi(x_n)+\frac{1}{2t_n}\sum\limits_{k=1}^{n-1}  \left(u_n \Phi(t_k x_k)u_n^{-1} - \Phi(t_k x_k)\right).
\end{align*}
By Lemma~\ref{constructiontk_x_k}, we get
\begin{align*}
&\left\|-\Phi(x_n)\right\|_{S(\cN)} \ge \varepsilon;\\
&\Big\|\frac{1}{2t_n}\sum\limits_{k=1}^{n-1}  \left(u_n \Phi(t_k x_k)u_n^{-1} - \Phi(t_k x_k)\right)\Big\|_{S(\cN)}  \le \frac{\varepsilon}{2}.
\end{align*}
Using the triangle inequality of $F$-norm,  we obtain that
\begin{align*}
\Big\|\frac{1}{2t_n}\Big(u_n \Phi(x) u_n^{-1}  -  \Phi(x)\Big)\Big\|_{S(\cN)}  \ge  \frac{\varepsilon}{2}.
\end{align*}
However,  using the symmetricity of $\left\|\cdot\right\|_{S(\cN)}$ and the triangle inequality, we obtain that
\begin{align*}
\Big\|\frac{1}{2t_n}\Big(u_n \Phi(x) u_n^\ast  -  \Phi(x)\Big)\Big\|_{S(\cN)}  \le
\Big\|\frac{1}{2t_n} \Phi(x)\Big\|_{S(\cN)}+\Big\|\frac{1}{2t_n} \Phi(x)\Big\|_{S(\cN)}\to 0,
\end{align*}
as  $t_n\to \infty.$
This contradiction implies  that $\Phi$ is real linear.
\end{proof}

\section{Continuity of ring isomorphisms in the measure topology}

In this Section we  assume that  $\Phi$ is a  real algebra isomorphism from $S(\mathcal{M})$ onto
$S(\cN)$ which is discontinuous in the measure topology. In order to come to a contradiction
we should slightly modify the proof from the previous Section.

\begin{lem}\label{dyadic}
There exists a non-zero projection $z\in  P(Z(\cN))$ with the following property: for any  projection $e\in P(z\cN)$ there exists a sequence
$\{x_n\}$ in $S(\cM)$ such that $x_n \stackrel{t_{\tau_\cM}}{\longrightarrow} 0$  and $\Phi(x_n)\stackrel{t_{\tau_\cN}}{\longrightarrow} e.$
\end{lem}

\begin{proof}
Consider the separating space of $\Phi$ which is defined as follows
$$
S(\Phi) =\{y \in S(\cN):~ \exists \{x_n\}\subset S(\cM),~
\left\|x_n \right\|_{S(\cN)} \to 0, ~\Phi(x_n) \stackrel{t_{\tau_\cN}}{\longrightarrow}  y \}.
$$
Due to the assumption of the discontinuity of $\Phi$ and by  the closed graph theorem (see \cite[Page 79]{Yos}), we have $S(\Phi)\neq\{0\}.$

Take a non zero element  $y\in S(\Phi).$
Let $i(y)$ be the partial inverse of $y,$  thus $yi(y)=l(y).$
Then for the sequence $\left\{x_n^{(1)}:= x_n\Phi^{-1}(i(y))\right\}$ we have that
\begin{eqnarray*}
&x_n^{(1)}=x_n\Phi^{-1}(i(y)) &
\stackrel{t_{\tau_\cM}}{\longrightarrow} 0
\end{eqnarray*}
and
\begin{eqnarray*}
& \Phi\left(x_n^{(1)}\right)& = \Phi\left(x_n\Phi^{-1}(i(y))\right)  = \Phi\left(x_n\right)i(y)
\stackrel{t_{\tau_\cN}}{\longrightarrow} y i(y)=l(y).
\end{eqnarray*}
There are a projection $p\le l(y)$ and $m\in \mathbb{N}$ such that
$
\Delta(p)=\frac{1}{m}z,
$
where $z=c(p)$ is the central support projection of $p.$
Now  for the sequence \linebreak  $\left\{x_n^{(2)}:= x_n^{(1)}\Phi^{-1}(p)\right\}$ we have that
\begin{eqnarray*}
&x_n^{(2)}&
\stackrel{t_{\tau_\cM}}{\longrightarrow} 0
\end{eqnarray*}
and
\begin{eqnarray*}
& \Phi\left(x_n^{(2)}\right)&
\stackrel{t_{\tau_\cN}}{\longrightarrow} p.
\end{eqnarray*}

Since $
\Delta(p)=\frac{1}{m}z,
$
there are mutually orthogonal equivalent projections $p_1=p,  \ldots, p_m\in P(\cN)$ such that
$\sum\limits_{i=1}^m p_i=z.$ Take partial isometries $w_1, \ldots, w_m\in \cN$ such that
$w_i^\ast w_i=p_1$ and $w_i w_i^\ast=p_i$ for all $i=1, \ldots, m.$
Further, for the sequence $\left\{x_n^{(3)}:= \sum\limits_{i=1}^m \Phi^{-1}(w_i) x_n^{(2)}\Phi^{-1}(w_i^\ast)\right\}$ we have that
\begin{eqnarray*}
&x_n^{(3)}&
\stackrel{t_{\tau_\cM}}{\longrightarrow} 0
\end{eqnarray*}
and
\begin{eqnarray*}
& \Phi\left(x_n^{(3)}\right)&
\stackrel{t_{\tau_\cN}}{\longrightarrow} z.
\end{eqnarray*}

Finally, for $e\in P(z\cN)$, setting
$\left\{x_n^{(4)}:= x_n^{(3)}\Phi^{-1}(e)\right\}$, we obtain that
\begin{eqnarray*}
&x_n^{(4)}&
\stackrel{t_{\tau_\cM}}{\longrightarrow} 0
\end{eqnarray*}
and
\begin{eqnarray*}
& \Phi\left(x_n^{(4)}\right)&
\stackrel{t_{\tau_\cN}}{\longrightarrow} e.
\end{eqnarray*}
\end{proof}

Let $z\in \cN$ be the central projection from Lemma~\ref{dyadic}.

Below in this section  replacing, if necessary, the algebras $\cM,$  $\cN$ respectively by $\Phi^{-1}(z)\cM, \ z\cN$ and replacing $\Phi$ by $\Phi|_{\Phi^{-1}(z)\cM}$, without loss of generality we can assume that $z=\mathbf{1}.$  Also we assume that $\tau_\cM$ is a normalised trace, that is, $\tau(\mathbf{1})=1.$

\begin{lemma}\label{zkzk}
For each $t_k\in \mathbb{R}$  $(k\ge 0)$
there exists an element $a_k \in  S(\mathcal{M})$ such that  the element $x_k$ defined by
\begin{equation}\label{def_x_k}
x_k  = \sum\limits_{i=0}^{2^{k}-1} (-1)^i v_k^{i} a_k   v_k^{-i}
\end{equation}
 satisfies the following conditions
\begin{itemize}
    \item[(i)] $\left\|t_k x_k\right\|_{S(\cM)}\le \frac{1}{2^k};$
    \item[(ii)] $\displaystyle \left\|\Phi(x_k)\right\|_{S(\cN)}\ge \frac{1}{2}.$
\end{itemize}
\end{lemma}

\begin{proof}
Let $t_k\in \mathbb{R}$ be  fixed. By Lemma~\ref{dyadic}
there exists  a sequence $\left\{a_n\right\}$ in $S(\mathcal{M})$ such that
$$
t_k a_n\stackrel{||\cdot||_{S(\cM)}}{\longrightarrow}  0,~ \mbox{ as }n\to \infty
$$
and
$$
\Phi\left(a_n\right)\stackrel{||\cdot||_{S(\cN)}}{\longrightarrow}  e_{1,1}^{(k)},~ \mbox{ as }n\to \infty .
$$
Since the sequence $\{a_n\}$ in
$S(\cM)$  converges to $0$ by the norm $ \left\|\cdot\right\|_{S(\cM)} ,$  it follows that
$$
t_k a a_n b \stackrel{||\cdot||_{S(\cM)}}{\longrightarrow}  0
$$
for all fixed  $a, b\in S(\cM)$ (see \cite[Section 2.4]{DP2} and Remark \ref{2.3}), and therefore
\begin{eqnarray*}
&t_k\sum\limits_{i=0}^{2^{k}-1} (-1)^i v_k^{i} a_n   v_k^{-i} & \stackrel{||\cdot||_{S(\cM)}}{\longrightarrow} 0,   ~\mbox{ as }n\to \infty
\end{eqnarray*}
and
\begin{eqnarray*}
&\Phi\Big(\sum\limits_{i=0}^{2^{k}-1} (-1)^i v_k^{i} a_n   v_k^{-i}\Big) & =
\sum\limits_{i=0}^{2^{k}-1} (-1)^i u_k^{i} \Phi(a_n)   u_k^{-i}\stackrel{||\cdot||_{S(\cN)}}{\longrightarrow}
 \sum\limits_{i=0}^{2^{k}-1} (-1)^i u_k^{i} e_{1, 1}^{(k)}   u_k^{-i}\\
 &=& \sum\limits_{i=1}^{2^{k}} (-1)^{i-1}e_{i, i}^{(k)},~\mbox{ as }n\to \infty  .
\end{eqnarray*}
 Since $\tau_\cM$ is a normalised trace, by the definition of the $F$-norm $||\cdot||_{S(\cN)}$ (see Remark~\ref{2.3}), we obtain that $\left\|\sum\limits_{i=1}^{2^{k}} (-1)^{i-1}e_{i, i}^{(k)}\right\|_{S(\cN)}= \left\|\mathbf{1}\right\|_{S(\cM)}= 1.$ Therefore  there exists an integer  $n$ such that
$$
\Big|\Big|t_k\sum\limits_{i=0}^{2^{k}-1} (-1)^i v_k^{i} a_n   v_k^{-i} \Big|\Big|_{S(\cM)}\le \frac{1}{2^k}
$$
and
$$
\left\|\Phi\Big(\sum\limits_{i=0}^{2^{k}-1} (-1)^i v_k^{i} a_n   v_k^{-i} \Big)\right\|_{S(\cN)} \ge \frac{1}{2}
$$
as required.
\end{proof}

The following lemma  is a technical improvement of Lemma \ref{zkzk} and the last step to the proof of   Theorem \ref{realisomorphism}.

\begin{lemma}\label{mmmm}
There exist a  sequence $\left\{x_k\right\}_{k\ge0}$ in $S(\cM)$
 and an increasing sequence
 $\{t_k\}_{k\ge 0}$  in $\mathbb{R}$ with $ 1=t_0$ and $t_k \to \infty$  such that
    \begin{itemize}
        \item[(1)] each $x_k$ is defined as in \eqref{def_x_k};
        \item[(2)] $\displaystyle \left\|t_kx_k\right\|_{S(\cM)} \le  \frac{1}{2^k};$
        \item[(3)] $\displaystyle \left\|\Phi(x_k)\right\|_{S(\cN)}\ge \frac{1}{2};$
         \item[(4)] $\displaystyle \left\|\frac{1}{2t_k}\sum\limits_{i=0}^{k-1}u_k\Phi(t_i x_i)u_k^{-1}-\Phi(t_i x_i)\right\|_{S(\cN)} \le  \frac{1}{4}.$
    \end{itemize}
    \end{lemma}

\begin{proof}
The proof is by induction on $k.$ For $k=0,$ since $u_0=\mathbf{1}$ and $t_0=1,$ it follows directly from Lemma~\ref{zkzk}.

Now assume that we have constructed $x_0, \ldots, x_k$ and $t_0, \ldots, t_k$ with properties (1)-(4).

Due to the definition of $F$-norms, we have
$$
\lim_{t \rightarrow \infty}\left\|\frac{1}{2t}\sum\limits_{i=0}^{k}u_{k+1}\Phi(t_i x_i)u_{k+1}^{-1}-\Phi(t_ix_i)\right\|_{S(\cN)}= 0.
$$
Hence,
we can take $t_{k+1}>t_k+n$ such that
$$
\left\|\frac{1}{2t_{k+1}}\sum\limits_{i=0}^{k}u_{k+1}\Phi(t_ix_i)u_{k+1}^{-1}-\Phi(t_ix_i)\right\|_{S(\cN)} \le  \frac{1}{4}.
$$

Further, for $t_{k+1},$  appealing to   Lemma~\ref{zkzk}, we obtain   $x_{k+1}$ as in \eqref{def_x_k} such that
$$
\displaystyle \left\|t_{k+1}x_{k+1}\right\|_{S(\cM)} \le  \frac{1}{2^{k+1}}
$$
and
$$
\displaystyle \left\|\Phi(x_{k+1})\right\|_{S(\cN)}\ge \frac{1}{2}.
$$
Now, the sequences $\{x_k\}_{k=0}^\infty$ and $\{t_k\}_{k=0}^\infty $
satisfy the properties claimed in the lemma.
\end{proof}

Now, we are ready to proceed to the proof of   Theorem \ref{realisomorphism} for a type II$_1$ von Neumann algebra  $\cM$ with a faithful normal finite trace $\tau$ ( $\tau(\mathbf{1})=1$).

\begin{proof}[Proof of Theorem \ref{realisomorphism}]

\

Let $x_0,\ldots, x_k,\ldots$ and $t_0,\ldots, t_k,\ldots$ from Lemma~\ref{mmmm}. Set
$$
x=\sum\limits_{k=1}^\infty y_k,
$$
where $y_k=t_kx_k.$ Lemma~\ref{mmmm} (2) guarantees that the above series is $\|\cdot\|_{S(\cM)}$-norm convergent in $S(\cM).$

As in the proof of Theorem~\ref{ringisomorphism} using Lemma \ref{lemma3.4} we have that
\begin{align*}
u_n \Phi(x) u_n^{-1}  -  \Phi (x)  = &  \Phi\left(v_n x v_n^{-1}-x\right)=
-2\Phi(y_n)  + \sum\limits_{k=1}^{n-1}  \left(u_n \Phi(y_k) u_n^{-1}- \Phi(y_k)\right).
\end{align*}
Recalling that $y_k=t_kx_k$, we have
\begin{align*}
\frac{1}{2t_n}\left(u_n \Phi(x) u_n^{-1}  -  \Phi(x)\right) & = -\Phi(x_n)+\frac{1}{2t_n}\sum\limits_{k=1}^{n-1}  \left(u_n \Phi(t_k x_k)u_n^{-1} - \Phi(t_k x_k)\right).
\end{align*}
By Lemma~\ref{mmmm}, we get
\begin{align*}
&\left\|-\Phi(x_n)\right\|_{S(\cN)} \ge \frac{1}{2};\\
&\Big\|\frac{1}{2t_n}\sum\limits_{k=1}^{n-1}  \left(u_n \Phi(t_k x_k)u_n^{-1} - \Phi(t_k x_k)\right)\Big\|_{S(\cN)}  \le \frac{1}{4}.
\end{align*}
Applying the  triangle inequality for $F$-norms,  we obtain that
\begin{align*}
\Big\|\frac{1}{2t_n}\Big(u_n \Phi(x) u_n^{-1}  -  \Phi(x)\Big)\Big\|_{S(\cN)} & \ge  \frac{1}{4}.
\end{align*}
On the other hand,  using the symmetricity of the norm  $\left\|\cdot\right\|_{S(\cN)}$ and the triangle inequality, we obtain that
\begin{align*}
\Big\|\frac{1}{2t_n}\Big(u_n \Phi(x) u_n^\ast  -  \Phi(x)\Big)\Big\|_{S(\cN)} & \le
\Big\|\frac{1}{2t_n} \Phi(x)\Big\|_{S(\cN)}+\Big\|\frac{1}{2t_n} \Phi(x)\Big\|_{S(\cN)}\to 0,
\end{align*}
as  $t_n\to \infty.$
From this contradiction we conclude  that $\Phi$ is continuous in the measure topology.
\end{proof}

\section{General form of ring isomorphisms}

In this Section we shall prove Theorem~\ref{latticeisomorphism} which is the main result of the paper.

Let   $\mathcal{M}$ and $\cN$ be  arbitrary type II$_1$ von Neumann algebras with  faithful normal finite traces $\tau_{\cM}$ and
$\tau_{\cN},$ respectively and let $\Phi:\cM \to \cN$ be a ring isomorphism, which  is a continuous real algebra isomorphism according to Theorems~\ref{ringisomorphism} and~\ref{realisomorphism}.

\begin{lemma}\label{orto}  Let   $p_\cN\in P(\cN)$ be a projection.  Suppose that  $x\in \cM$  and
$q_\cM=\mathbf{1}- s\left(\Phi^{-1}\left(\Phi(x)^\ast\right)\right) \vee
s\left(\Phi^{-1}(p_\cN)\right),$ where
$s(a)$ denotes the support of an element $a.$ Then
\begin{equation*}
(p_\cN+q_\cN)\Phi(x+y)^\ast \Phi(x+y)(p_\cN+q_\cN)= p_\cN\Phi(x)^\ast \Phi(x)p_\cN+q_\cN\Phi(y)^\ast \Phi(y)q_\cN
\end{equation*}
for all $y\in q_\cM\cM q_\cM,$ where $q_\cN\in \cN$ is an arbitrary projection with $q_\cN\le r(\Phi(q_\cM)).$
\end{lemma}

\begin{proof} Let $y\in q_\cM \cM q_\cM$ be an arbitrary element.

Firstly, from
$s\left(\Phi^{-1}\left(\Phi(x)^\ast\right)\right)y=0,$
it follows that $\Phi^{-1}\left(\Phi(x)^\ast\right)y=0.$
Thus
$$
\Phi(x)^\ast \Phi(y)=\Phi\Big(\Phi^{-1}\left(\Phi(x)^\ast\right)y\Big)=\Phi(0)=0
$$
and
$$
\Phi(y)^\ast \Phi(x)=\Big(\Phi(x)^\ast\Phi(y)\Big)^\ast=0.
$$
Hence
\begin{align*}
\Phi(x+y)^\ast \Phi(x+y) =  \Phi(x)^\ast \Phi(x) + \Phi(y)^\ast \Phi(y).
\end{align*}

Secondly, since $q_\cM s\left(\Phi^{-1}(p_\cN)\right)=0,$ it follows that $y\Phi^{-1}(p_\cN)=0,$ and therefore
$$
\Phi(y)p_\cN=0 = p_\cN\Phi(y)^\ast.
$$

Finally, from
$q_\cM s\left(\Phi^{-1}\left(\Phi(x)^\ast\right)\right)=0,$
it follows that $q_\cM \Phi^{-1}\left(\Phi(x)^\ast\right)=0.$
Thus
$
\Phi(q_\cM)\Phi(x)^\ast  =0,
$
and therefore
$$
\Phi(x) \Phi(q_\cM)^\ast=0.
$$
Multiplying the partial inverse of
$\Phi(q_\cM)^\ast$ to the right of the
last equality we obtain that
$$
\Phi(x) r(\Phi(q_\cM))=0,
$$
in particular,
$$
\Phi(x) q_\cN=0,
$$
because $q_\cN\le r(\Phi(q_\cM)).$

So,
\begin{align*}
\Phi(x)^\ast \Phi(y) =  \Phi(y)^\ast \Phi(x) =0,\\
\Phi(x) q_\cN =  q_\cN \Phi(x)^\ast=0,\\
\Phi(y) p_\cN =  p_\cN \Phi(y)^\ast= 0.
\end{align*}
 Taking into account these equalities we get
\begin{eqnarray*}
 (p_\cN+q_\cN)\Phi(x+y)^\ast \Phi(x+y)(p_\cN+q_\cN) &=&
 (p_\cN+q_\cN)\Big(\Phi(x)^\ast\Phi(x)+\Phi(y)^\ast \Phi(y)\Big) 
 (p_\cN+q_\cN)\\
& = &   p_\cN\Phi(x)^\ast\Phi(x)p_\cN+q_\cN\Phi(y)^\ast \Phi(y)q_\cN
\\
& + &  p_\cN\Phi(y)^\ast\Phi(y)p_\cN+q_\cN\Phi(x)^\ast \Phi(x)q_\cN \\
& + &   p_\cN\Phi(x)^\ast\Phi(x)q_\cN+q_\cN\Phi(y)^\ast \Phi(y)p_\cN
  \\
& + &   p_\cN\Phi(y)^\ast\Phi(y)q_\cN+q_\cN\Phi(x)^\ast \Phi(x)p_\cN\\
& = &   p_\cN\Phi(x)^\ast\Phi(x)p_\cN+q_\cN\Phi(y)^\ast \Phi(y)q_\cN.
  \end{eqnarray*}
The proof is complete. \end{proof}

In the next Lemma we shall  use the following order on $S(\cM).$  For $x, y\in S(\cM)$ set
$$
x\prec y \Longleftrightarrow s(x)\le s(y),\, y=x+z,\, s(x)s(z)=0.
$$
Direct computations show that $\prec$ is a partial order on $S(\cM).$

The
following lemma is  one of the key steps towards the  proof of the main
result.

\begin{lemma}\label{boun}  There exists a sequence of projections
$\left\{q_n\right\}$ in  $\cM$ with \linebreak $\tau_\cM\left(\mathbf{1}-q_n\right)\rightarrow 0$ such that
$\Phi$ maps $q_n\cM q_n$ into $\cN$.
\end{lemma}

\begin{proof}
For every $n\in \mathbb{N}$ denote by $\mathcal{F}_n$ the set of all pairs $(x, p_\cN)\in \cM\times P(\cN)$ such that
\begin{itemize}
\item $||x||_\cM\le 1;$
\item $\tau_\cM\left(s(x)\right)\le 2\tau_\cM(l\left(\Phi^{-1}(p_\cN)\right));$
\item $p_\cN\Phi(x)^\ast \Phi(x)p_\cN\ge n p_\cN.$
\end{itemize}

Recall that $\|\cdot\|_\cM$ is the operator norm on $\cM.$

Note that $(0,0)\in \mathcal{F}_n,$ so $\mathcal{F}_n$ is not empty.
Let us show that the set $\mathcal{F}_n$ has a maximal element
with respect to the order $\leq,$ where
$$
(x_1, p_1)\le (x_2, p_2)\Leftrightarrow x_1\prec x_2,\, p_1\le p_2.
$$

Let $\{(x_\alpha, p_\alpha)\}\subset \mathcal{F}_n$ be a totally ordered net.
Since $\{s(x_\alpha)\}$ is an increasing net of projections from $\cM,$
it follows that $s(x_\alpha)$ converges in the strong operator topology to some projection  $s$ from $\cM.$
Then for $\alpha>\beta$ we have that
$$
\tau_\cM(s(x_\alpha-x_\beta))=\tau_\cM(s(x_\alpha)-s(x_\beta))\le \tau_\cM(s-s(x_\beta))\to 0.
$$
Thus  the  net $\{x_\alpha\}$ converges to some element $x$ from
the unit ball of $\cM$ in the measure topology,
moreover, $x=\sup\limits_\alpha x_\alpha,$ (here the least upper bound is taken with
respect to the above partial order $\prec$) and $s(x)=\sup\limits_\alpha s(x_\alpha)=s.$
Since $\{p_\alpha\}$ is also an increasing net of projections from
$\cN,$ then $p_\alpha \uparrow p,$ where $p\in P(\cN),$ in
particular, $p_\alpha\stackrel{t_{\tau_\cN}}\longrightarrow p.$

Now we check  that  $(x, p) \in \mathcal{F}_n.$
From
$x_\alpha\stackrel{t_{\tau_\cM}}\longrightarrow x,$ by  continuity of $\Phi$ we have that
$\Phi(x_\alpha)\stackrel{t_{\tau_\cN}}\longrightarrow \Phi(x).$
Let $\beta$ be a fixed index and take an arbitrary index $\alpha\ge \beta.$ Taking into
account that $p_\alpha\Phi(x_\alpha)^\ast \Phi(x_\alpha)p_\alpha\ge
n p_\alpha$ and $p_\alpha\ge p_\beta,$ we obtain that
$p_\beta\Phi(x_\alpha)^\ast \Phi(x_\alpha)p_\beta\ge
n p_\beta.$ Since $\Phi(x_\alpha)\stackrel{t_{\tau_\cN}}\longrightarrow \Phi(x),$ it follows that
$p_\beta\Phi(x)^{\ast}\Phi(x)p_\beta\geq n
p_\beta.$ From $p_\beta\uparrow p,$ we have that
$$
p\Phi(x)^{\ast}\Phi(x)p\geq n
p.
$$
Finally, since $\Phi^{-1}$ is continuous, $\tau_\cM(s(x_\alpha))\le 2
\tau_\cM(l\left(\Phi^{-1}(p_\alpha)\right)),$ $s(x_\alpha)\uparrow
s(x)$ and  $p_\alpha\uparrow p,$ it follows that
$\tau_\cM(s(x))\le 2 \tau_\cM(l\left(\Phi^{-1}(p)\right)).$ This
means that $(x,p)\in \mathcal{F}_n.$ Therefore, any totally
ordered net  in $\mathcal{F}_n$ has the least upper bound. By
Zorn's Lemma $\mathcal{F}_n$ has a maximal element, say $(x_n,
p_n).$

Put
$$
q_n=\mathbf{1}- s\left(\Phi^{-1}\left(\Phi(x_n)^\ast\right)\right) \vee
s\left(\Phi^{-1}(p_n)\right)\vee
s\left(x_n\right).
$$

Let us prove that
$$
\Phi(x)^\ast \Phi(x)\leq n \overline{q_n}
$$
for all $x\in q_n\cM q_n$ with $||x||_\cM\le 1,$ where $\overline{q_n}=r\left(\Phi\left(q_n\right)\right).$

The case  $q_n=0$  is trivial.

Let us consider the case $q_n\neq 0.$ Take a non zero element $x\in
q_n \cM q_n$ such that  $||x||_\cM\le 1.$ Note that $\Phi(x)^\ast\Phi(x)\in S(\overline{q_n}\cN \overline{q_n}),$ because
$\overline{q_n}=r\left(\Phi\left(q_n\right)\right)$ and $x\in q_n\cM q_n.$
Let $\Phi(x)^\ast\Phi(x)=\int\limits_{0}^{+\infty}\lambda \, d\,
e_{\lambda}$ be the spectral resolution of $\Phi(x)^\ast\Phi(x)$ in $S(\overline{q_n}\cN \overline{q_n}).$
Assume that $e=\overline{q_n}-e_{n}\neq 0,$ that is,
$$
e \Phi(x)^\ast\Phi(x) e\geq n e.
$$
Take  the projection
$$
f_n=l(q_n\Phi^{-1}(e))\le q_n.
$$
By the definition of the left support we have $0=(q_n-f_n)q_n\Phi^{-1}(e)=(q_n-f_n)\Phi^{-1}(e).$ Thus
\begin{eqnarray*}
0 & = &\Phi((q_n-f_n)\Phi^{-1}(e))=\Phi\left(q_n-f_n\right)e.
\end{eqnarray*}
Denote
$$
y=xf_n.
$$
Then $y\in q_n \cM q_n$ and $||y||_\cM\le 1.$ From
$$
\Phi(x(q_n-f_n))e=\Phi(x)\Phi(q_n-f_n)e=0,
$$ it follows that
$
\Phi(y)e=\Phi(x)e.
$
Thus
$$
e \Phi(y)^\ast\Phi(y) e=e \Phi(x)^\ast\Phi(x) e \geq n e.
$$
Further, $\tau_\cM(s(y))\le 2 \tau_\cM(l(\Phi^{-1}e)),$ because
    \begin{align*}
    \tau_\cM(s(y)) & =\tau_\cM(l(y)\vee r(y)) \le \tau_\cM(l(y))+\tau_\cM(r(y))=[l(y)\sim r(y)]\\
    & =2\tau_\cM(r(y))=2\tau_\cM(r(xf_n)) \le 2\tau_\cM(f_n)=2\tau_\cM(l(q_n\Phi^{-1}(e)))\\
    &=2\tau_\cM(r(q_n\Phi^{-1}(e)))\le 2\tau_\cM(r(\Phi^{-1}(e)))=2\tau_\cM(l(\Phi^{-1}(e))).
    \end{align*}

Applying  Lemma~\ref{orto} we have
\begin{eqnarray*}
 (p_n+e)\Phi(x_n+y)^\ast \Phi(x_n+y)(p_n+e) & = & p_n\Phi(x_n)^\ast\Phi(x_n)p_n+e\Phi(y)^\ast \Phi(y)e\\
 &\ge& n(p_n+e).
\end{eqnarray*}
Further, $||x_n+y||\le 1,$ because $||x_n||_\cM, ||y||_\cM\le 1,$ $s(y)s(x_n)=0.$ In particular, $x_n \prec x_n+y.$

Let us show that $\tau_\cM(s(x_n+y))\le 2
\tau_\cM(l(\Phi^{-1}(p_n+e))).$ Indeed, since $e\le
\overline{q_n}=r\left(\Phi(q_n)\right)=l\left(\Phi(q_n)^\ast\right),$
it follows that $e\Phi(q_n)=e.$ Further,
$\Phi^{-1}(e)s\left(\Phi^{-1}(p_n)\right)=0,$ because by the
choice  of $q_n$ we have that $q_n\le
\mathbf{1}-s\left(\Phi^{-1}(p_n)\right).$ Thus
$r\left(\Phi^{-1}(e)\right)r\left(\Phi^{-1}(p_n)\right)=0,$ and
therefore
$r\left(\Phi^{-1}(e+p_n)\right)=r\left(\Phi^{-1}(e)\right)+r\left(\Phi^{-1}(p_n)\right).$
From the last equality we obtain that
\begin{align*}
\tau_\cM(s(x_n+y))& \le  \tau_\cM(s(x_n))+\tau_\cM(s(y))\le 2\tau_\cM(l(\Phi^{-1}(p_n)))+2 \tau_\cM(l(\Phi^{-1}(e)))\\
& = 2\tau_\cM(r(\Phi^{-1}(p_n)))+ 2\tau_\cM(r(\Phi^{-1}(e)))\\
&= 2\tau_\cM(r(\Phi^{-1}(p_n+e)))=2\tau_\cM(l(\Phi^{-1}(p_n+e))).
\end{align*}

So, we have that $(x_n+y, p_n+e)\in \mathcal{F}_n$ and $(x_n,
p_n)\le (x_n+y, p_n+e).$ This contradicts  maximality of
$(x_n,p_n).$ This contradiction implies that
$\overline{q_n}-e_{n}=e=0.$ This means that
$$
\Phi(x)^\ast\Phi(x)\le n\overline{q_n}
$$
for all $x\in q_n\cM q_n$ with $||x||_\cM\le 1.$ In particular, $\Phi$ maps $q_n \cM q_n$ into $\cN.$

Let us show that $\tau_\cM\left(\mathbf{1}-q_n\right)\rightarrow 0.$
Consider the sequence $\left\{\frac{1}{\sqrt{n}}x_n\right\},$ where $(x_n, p_n)$ is a maximal element of $\mathcal{F}_n.$ From
$\frac{1}{\sqrt{n}}x_n\to 0,$ it follows that $\frac{1}{\sqrt{n}}\Phi(x_n) \to 0.$ Thus
$
p_n\to 0,$
because $\frac{1}{n}p_n\Phi(x_n)^\ast \Phi(x_n)p_n\ge p_n.$
Further, the continuity of $\Phi^{-1}$ implies that $\Phi^{-1}(p_n)\to 0$ in the measure topology. Since  each $\Phi^{-1}(p_n),\, n\in \mathbb{N},$ is an idempotent, by Lemma~\ref{limitofrang}, we obtain that
$l\left(\Phi^{-1}(p_n)\right)\to 0$ and $s\left(\Phi^{-1}(p_n)\right)\to 0$ in the measure topology.
Further, the inequality
$\tau_\cM\left(s(x_n)\right)\le 2\tau_\cM\left(l(\Phi^{-1}p_n)\right)$ implies that
$s(x_n)\to 0$ in the measure topology. Thus
$x_n\to 0$ in the measure topology.
Since
$$
s\left(\Phi^{-1}\left(\Phi(x_n)^\ast\right)\right)=
s\left(\Phi^{-1}\left(\Phi(s(x_n))^\ast\right)\Phi^{-1}\left(\Phi(x_n)^\ast\right)
\right)\le s\left(\Phi^{-1}\left(\Phi(s(x_n))^\ast\right)\right)
$$
and $\Phi^{-1}\left(\Phi(s(x_n))^\ast\right)$ is an idempotent, it follows that
$
s\left(\Phi^{-1}\left(\Phi(x_n)^\ast\right)\right)\to 0.
$
Hence,
$$
s\left(\Phi^{-1}\left(\Phi(x_n)^\ast\right)\right) \vee
s\left(\Phi^{-1}(p_n)\right)\vee
s\left(x_n\right)\to 0,
$$
and hence
$\tau_\cM\left(\mathbf{1}-q_n\right)\rightarrow 0.$
The proof is complete.
\end{proof}

In the next two Lemmas we assume that $\Phi$ is a real algebra isomorphism of $S(\cN)$ onto itself (i.e. real automorphism).

\begin{lemma}\label{aaa}
Let $e\in \cN$ be a projection with $e\sim \mathbf{1}-e=f$ and let $\Phi$ be a real automorphism of $S(\cN)$ such that
$\Phi$ acts on $e\cN e$ identically. Then there is an invertible element $a\in S(\cN)$ such that
$$
\Phi(x)=axa^{-1}
$$
for all $x\in S(\cN).$
\end{lemma}

\begin{proof}
Since $\Phi(e)=e,$ it follows that $\Phi(f)=f.$
Let $u\in \cN$ be a partial isometry such that $u^\ast u=e$ and $u u^\ast=f.$
Then $a=e+\Phi(u)u^\ast$ is invertible and $a^{-1}=e+u\Phi(u^\ast ).$ Indeed, taking into account that
$ueu^\ast = f, u^2=0,$ we obtain that
\begin{eqnarray*}
(e+\Phi(u)u^\ast)(e+u \Phi(u^\ast)) & = & e^2+\Phi(u)u^\ast u \Phi(u^\ast) = e^2+\Phi(u) e \Phi(u^\ast)\\
&=& e + \Phi(ueu^\ast) = e + \Phi(f)=\mathbf{1}.
\end{eqnarray*}
Likewise $(e+u \Phi(u^\ast))(e+\Phi(u)u^\ast)=\mathbf{1}.
$

Recall that $e\cN e$ is dense in $S(e\cN e)$ in the measure topology. Since $\Phi$ is continuous in this topology  and acts on $e\cN e$ identically, it follows that $\Phi$ also acts identically on $S(e\cN e).$
Therefore, $\Phi(x)=x=axa^{-1}$ for all $x\in S(e\cN e),$ because $a=e+\Phi(u)u^\ast\in e\cN e+ S(f\cN f).$

Further for any $x\in S(f\cN f)$ we have that
\begin{eqnarray*}
axa^{-1} &=& (e+\Phi(u)u^\ast) x (e+u\Phi(u^\ast)) = (e+\Phi(u)u^\ast) f x f(e+u\Phi(u^\ast))\\
&=&
\Phi(u)u^\ast fxf u\Phi(u^\ast)=\Phi(u)u^\ast x u\Phi(u^\ast)= [u^\ast x u\in S(e\cN e)]\\
& =& \Phi(u)\Phi(u^\ast x u)\Phi(u^\ast)= \Phi(u u^\ast x u u^\ast)=\Phi(fxf)=\Phi(x).
\end{eqnarray*}
Similarly, for $x\in eS(\cN)f$ or $x\in fS(\cN)e$ we also have that
$\Phi(x)=axa^{-1}.$
\end{proof}

\begin{lemma}\label{bbb}
Let $e\in \cN$ be a projection with $e\sim \mathbf{1}-e=f$ and let $\Phi$ be a real automorphism of $S(\cN)$ such that
$l(\Phi(e))=e$ and $\Phi(x)l(\Phi(e))=x$ for all $x\in S(e\cN e).$ Then there is an invertible element $b\in S(\cN)$ such that
$$
\Phi(x)=bxb^{-1}
$$
for all $x\in S(e\cN e).$
\end{lemma}

\begin{proof}
Since $\Phi(e)$ is an idempotent, by \eqref{rangeofidem} the element
$e=l(\Phi(e))$ is the range projection of  $\Phi(e),$ that is
$$
e\Phi(e)=\Phi(e),\,\, \Phi(e)e=e.
$$
Then $\Phi(e)=e+w,$ where $w\in eS(\cN)f.$

It is clear that $b=\mathbf{1}-w$ is invertible and $b^{-1}=\mathbf{1}+w,$ because $w^2=0.$

Now for $x\in S(e\cN e)$ we have that
\begin{eqnarray*}
bxb^{-1} &=& (\mathbf{1}-w)x(\mathbf{1}+w) =  x(\mathbf{1}+w) \\
&=&
\Phi(x)e(\mathbf{1}+w)=\Phi(x)e+\Phi(x)w=\Phi(x)e+\Phi(x)(\Phi(e)-e)\\
& =& \Phi(x)e+\Phi(xe)-\Phi(x)e=\Phi(x).
\end{eqnarray*}
\end{proof}

\begin{lemma}\label{eqi}
Let $p\in \cM$ be a projection with $p\sim \mathbf{1}-p,$  let $\Phi$ be a real algebra  isomorphism from $S(\cM)$ onto $S(\cN)$ and let  $e=l\left(\Phi(p)\right)$ be the range projection of $\Phi(p).$ Then $e\sim \mathbf{1}-e.$

\end{lemma}

\begin{proof}
Since $p\sim \mathbf{1}-p,$ there exists a partial isometry in $\cM$ such that
$u^\ast u=p$ and $\mathbf{1}-p=uu^\ast.$ For convenience denote $a=\Phi(u^\ast)$ and $b=\Phi(u).$
Then
\begin{align*}
ab+ba &= \Phi(u^\ast)\Phi(u)+\Phi(u)\Phi(u^\ast)=\mathbf{1}
\end{align*}
and
\begin{align*}
l(ab) &= l\left(\Phi(u^\ast)\Phi(u)\right)=l\left(\Phi(u^\ast u)\right)=l\left(\Phi(p)\right)=e.
\end{align*}
Note that $eab=ab$ and $abe=e,$ because $e$ is the range projection of the idempotent $\Phi(p).$
From these equalities we obtain that
\begin{align*}
ba(\mathbf{1}-e)& =(\mathbf{1}-ab)(\mathbf{1}-e)=\mathbf{1}-e-ab+abe=\mathbf{1}-ab=ba.
\end{align*}
Thus
\begin{equation}\label{bae}
r(ba)\precsim \mathbf{1}-e.
\end{equation}

Since
$u$ is a partial isometry, it follows that
$u=uu^\ast u$ and $u^\ast =u^\ast u u^\ast.$ Hence
$a=aba$ and $b=bab.$
Using the first equality we have that
\begin{align*}
l(a) & =l(aba)\precsim l(ab) \precsim r(b)\sim l(b).
\end{align*}
Likewise
\begin{align*}
l(b) & \precsim l(ba) \precsim r(a)\sim l(a).
\end{align*}
Since $\cM$ is of type II$_1,$ all projections in the last two relations are equivalent. Thus
\begin{align*}
l(a) & \sim l(ab) \sim l(ba)\sim l(b),
\end{align*}
and hence
\begin{align*}
e=l(ab) & \sim l(ba) \sim r(ba)\stackrel{\eqref{bae}}{\precsim}\mathbf{1}-e.
\end{align*}
Thus from the definition of the dimension function $\Delta_\cN$ on
$\cN$ (see 2.2) it follows that  $\Delta_\cN(e)\le
\Delta_\cN(\mathbf{1}-e)=\mathbf{1}-\Delta_\cN(e),$ and therefore
$\Delta_\cN(e)\le \frac{1}{2}\mathbf{1}.$

    On the other hand, since $\mathbf{1}=ab+ba,$ it follows that
    $$
    \mathbf{1}\le l(ab)\vee l(ba).
    $$
    Further
    \begin{align*}
    \mathbf{1}&=\Delta_\cN(\mathbf{1})\le \Delta_\cN(l(ab))+\Delta_\cN(l(ba))=\Delta_\cN(l(ab))+\Delta_\cN(l(ab))= 2\Delta_\cN(e),
    \end{align*}
    that is,
    $\Delta_\cN(e)\ge \frac{1}{2}\mathbf{1}.$ Hence
    $
    \Delta_\cN(e)=\frac{1}{2}\mathbf{1}=\Delta_\cN(\mathbf{1}-e).
    $
    Thus $e\sim \mathbf{1}-e.$ The proof is complete.
\end{proof}

Let $p\in \cM$ be a projection with $p\sim \mathbf{1}-p=q.$
Let us fix a partial isometry $v\in \cM$ such that
$vv^\ast=p$ and $v^\ast v=q.$ We shall identify the subsets $p\cM
q,$ $q\cM p,$ and $q\cM q$ with the
$p\cM  v,$ $v^\ast \cM p,$ and $v^\ast \cM v,$ respectively.
    Then the decomposition
    $
    x=x_{11}+x_{12}v+v^\ast x_{21}+v^\ast x_{22}v\in\cM,
    $
    where $x_{ij}\in p\cM p,$ $i,j=1,2,$ gives us a  representation of $\cM$ as a matrix algebra $M_2(p\cM p):$
    $$
    x\in \cM \to \left(
    \begin{array}{cc}
    x_{11} & x_{12} \\
    x_{21} & x_{22} \\
    \end{array}
    \right)\in M_2(p\cM p).
    $$

    Suppose that $\Psi$ is a  real $\ast$-isomorphism from $p \cM p$ onto $e \cN e,$
    where  $e\in \cN$ is a projection with $e\sim \mathbf{1}-e.$ Then $\Psi$ can be extended as a real $\ast$-isomorphism $\widetilde{\Psi}$ from $\cM\equiv M_2(p\cM p)$ onto $\cN\equiv M_2(e\cN e)$ as follows
\begin{equation}\label{ext}
    \widetilde{\Psi}\left(
    \begin{array}{cc}
        x_{11} & x_{12} \\
        x_{21} & x_{22} \\
    \end{array}
    \right)=\left(
    \begin{array}{cc}
        \Psi(x_{11}) & \Psi(x_{12}) \\
        \Psi(x_{21}) & \Psi(x_{22}) \\
    \end{array}
    \right).
\end{equation}

 Let $p\in \cM$ be an arbitrary  projection. Below,  for the sake of convenience, we denote by  $\widetilde{p}=l(\Phi(p))$ the range projection of the idempotent  $\Phi(p).$
Then the mapping $\Phi_{p,\widetilde{p}}:p\cM p \to S(\widetilde{p}\cN \widetilde{p})$ defined as
\begin{equation}\label{pe}
    \Phi_{p,\widetilde{p}}(x)=\Phi(x)\widetilde{p},\, x\in p \cM p
\end{equation}
is a real algebra homomorphism from $p\cM p$ into $S(\widetilde{p}\cN \widetilde{p}).$ Indeed,
\begin{eqnarray*}
    \Phi_{p,\widetilde{p}}(x)\Phi_{p,e}(y) &=& \Phi(x)\widetilde{p}\Phi(y)\widetilde{p} =\Phi(x)l(\Phi(p))\Phi(py)\widetilde{p} =
    \Phi(x)l(\Phi(p))\Phi(p)\Phi(y)\widetilde{p} \\
    &=&
    \Phi(x)\Phi(p)\Phi(y)\widetilde{p} =\Phi(xpy)\widetilde{p} =\Phi(xy)\widetilde{p} =\Phi_{p,\widetilde{p}}(xy).
\end{eqnarray*}

\begin{lemma}\label{ccc}
Let $p\in \cM$ be a projection with $p\sim \mathbf{1}-p$ and let $\Phi$ be a real algebra  isomorphism from $S(\cM)$ onto $S(\cN)$ such that
$\Phi_{p,\widetilde{p}}$  maps $p\cM p$ onto  $\widetilde{p}\cN \widetilde{p},$ where $\Phi_{p,\widetilde{p}}$ is defined as \eqref{pe}.
Then there are real $\ast$-isomorphism $\Psi$ from $S(\cM)$ onto $S(\cN)$  and an invertible element $c\in S(\cN)$ such that
$$
\Phi(x)=c\Psi(x)c^{-1}
$$
for all $x\in S(\cM).$
\end{lemma}

\begin{proof}
Since $\Phi_{p,\widetilde{p}}$
is a real algebra isomorphism from $p\cM p$ onto $\widetilde{p}\cN \widetilde{p},$
by \cite[Lemma 2.1 (3)]{MMori2020}
there are real $\ast$-isomorphism $\Psi$ from $p\cM p$ onto $\widetilde{p}\cN \widetilde{p}$  and an invertible element $d$ in $\widetilde{p}\cN \widetilde{p}$ such that
\begin{equation}\label{ped}
\Phi_{p,\widetilde{p}}(x)=d\Psi(x)d^{-1}
\end{equation}
for all $x\in p \cM p.$
Replacing, if necessary, $d$ with $d+\mathbf{1}-\widetilde{p},$ we may assume that $d$ is an invertible element in $\cN$ with the property \eqref{ped}. Note that by Proposition~\ref{eqi}, $\widetilde{p}\sim \mathbf{1}-\widetilde{p}.$ Thus by \eqref{ext} the   real $\ast$-isomorphism $\Psi$ can be extended from $\cM$ onto $\cN$ which we also denote as $\Psi.$
Further, by \cite[Theorem 1, p. 230]{SHAC}
 real $\ast$-isomorphism $\Psi,$ which is a direct sum of a $\ast$-isomorphism
and a conjugate-linear $\ast$-isomorphism  can be extended from $S(\cM)$ onto $S(\cN)$ which we still denote by $\Psi.$
Set
\begin{equation}\label{psi1}
\Phi_1(y)=\Phi(\Psi^{-1}(d^{-1}yd)),\,\, y\in S(\cN).
\end{equation}
Then $\Phi_1$ is a real  automorphism of $S(\cN)$ such that
$\Phi_1(y)\widetilde{p}=y$ for all $y\in S(\widetilde{p}\cN \widetilde{p}).$
Indeed, for $y\in S(\widetilde{p}\cN \widetilde{p})$ we obtain that
\begin{eqnarray*}
\Phi_1(y)\widetilde{p}  & = & \Phi(\Psi^{-1}(d^{-1}yd))\widetilde{p} \stackrel{\eqref{pe}}= \Phi_{p,\widetilde{p}}(\Psi^{-1}(d^{-1}yd))\\
&\stackrel{\eqref{ped}}= & d\Psi\left(\Psi^{-1}(d^{-1}yd)\right)d^{-1} =
y.
\end{eqnarray*}

By Lemma~\ref{bbb} we can find
an invertible element $h\in S(\cN)$ such that
$$
\Phi_1(y)=hyh^{-1},
$$
or
$$
h^{-1}\Phi_1(y) h=y
$$
for all $y\in e\cN e.$

Finally,
by Lemma~\ref{aaa} there exists an invertible element $a\in S(\cN)$ such that
$$
h^{-1}\Phi_1(y)h=aya^{-1}
$$
or
\begin{equation}\label{ha1}
\Phi_1(y)=(ha)y(ha)^{-1}
\end{equation}
for all $y\in S(\cN).$ For an arbitrary $x\in S(\cM)$ putting $y=d\Psi(x)d^{-1}$ and combining \eqref{psi1}  with \eqref{ha1} we obtain that
$$
\Phi(x)=c\Psi(x)c^{-1},
$$
where $c=had.$
\end{proof}

 We need also the following auxiliary result.

\begin{lemma}\label{boune}  There exists a sequence of projections
    $\left\{p_n\right\}$ in  $\cM$ with \linebreak $\tau_\cM\left(\mathbf{1}-p_n\right)\rightarrow 0$ such
    that $\Phi_{p_n, \widetilde{p_n}}$ maps $p_n\cM p_n$ onto $\widetilde{p_n} \cN \widetilde{p_n}$ for all $n\in \mathbb{N},$ where $\Phi_{p_n,\widetilde{p_n}}$ is defined by \eqref{pe}.
\end{lemma}

\begin{proof}
By Lemma~\ref{boun} there exist sequences of projections    $\{q_n\}$ and $\{g_n\}$ in $\cM$ and $\cN,$ respectively, such that
\begin{itemize}
\item $\Phi$ maps $q_n \cM q_n$ into $\cN;$
\item $\Phi^{-1}$ maps $g_n \cN g_n$ into $\cM;$
\item $\tau_\cM\left(\mathbf{1}-q_n\right)\rightarrow 0$  and $\tau_\cN\left(\mathbf{1}-g_n\right)\rightarrow 0.$
\end{itemize}
Consider the following projections
$$
p_n = l\left(\Phi^{-1}(g_n)\right)\wedge q_n\in P(\cM)
$$
and
$$
\widetilde{p_n}=l(\Phi(p_n))
$$
for each $n\in \mathbb{N}.$ According to Lemma~\ref{limitofrang}, from  $\tau_\cM\left(\mathbf{1}-q_n\right)\rightarrow 0$  and $\tau_\cN\left(\mathbf{1}-g_n\right)\rightarrow 0,$ we obtain  that $\tau_\cM\left(\mathbf{1}-p_n\right)\rightarrow 0.$
Since $p_n \le  l\left(\Phi^{-1}(g_n)\right)$ and also  $l\left(\Phi^{-1}(g_n)\right)$ is a range projection of $\Phi^{-1}(g_n),$ it follows that
$$
p_n=l(\Phi^{-1}(g_n))p_n=\Phi^{-1}(g_n)l(\Phi^{-1}(g_n))p_n=
\Phi^{-1}(g_n)p_n.
$$
Thus
$$
\Phi(p_n)=g_n\Phi(p_n),
$$
and therefore
$$
\widetilde{p_n}=l(\Phi(p_n))\le g_n.
$$

Consider the element  $y=\widetilde{p_n}y\widetilde{p_n}\in \widetilde{p_n}\cN \widetilde{p_n}\subset g_n \cN g_n.$ By the choice of the projections $g_n,$ there exists an element $x\in \cM$ such that
$\Phi(x)=y.$ Set
$$
x_1=p_nxp_n\in p_n\cM p_n.
$$
Then
\begin{align*}
\Phi_{p_n,\widetilde{p_n}}(x_1) & =\Phi(p_nxp_n)\widetilde{p_n}=\Phi(p_n)\Phi(x)\Phi(p_n)\widetilde{p_n}=\Phi(p_n)\Phi(x)\widetilde{p_n}\\
&
=\Phi(p_n)y\widetilde{p_n}=\Phi(p_n)\widetilde{p_n}y\widetilde{p_n}=\widetilde{p_n}y\widetilde{p_n}=y.
\end{align*}
So, $\Phi_{p_n,\widetilde{p_n}}$ maps $p_n \cM p_n$ onto $\widetilde{p_n} \cN \widetilde{p_n}.$ The proof is complete.
\end{proof}

\begin{proof}[Proof of Theorem \ref{latticeisomorphism}]

\

By Lemma~\ref{boune}  there exists a sequence of projections
$\{p_n\}$ in $\cM$ such that
\begin{itemize}
\item $\tau_\cM(\mathbf{1}-p_n)\rightarrow 0;$
\item $\Phi_{p_n, \widetilde{p_n}}$ maps $p_n \cM p_n$ onto $\widetilde{p_n} \cN \widetilde{p_n}$ for all $n\in \mathbb{N},$
\end{itemize}
where $\Phi_{p_n, \widetilde{p_n}}$ are defined as in \eqref{pe}.

Set
$$
z_1=\sup\left\{z\in P(Z(\cM)): z\Delta(p_1)\ge \frac{1}{2}z\right\}.
$$
Then $\Delta(z_1p_1)\ge \frac{1}{2}z_1.$

Let $n\ge 2.$ Assume that we have constructed mutually orthogonal central projections $z_1, \ldots, z_{n-1}$ in $\cM$ such that
$\Delta(z_kp_k)\ge \frac{1}{2}z_k$ for all $k=1, \ldots, n-1.$
Setting
$$
z_n=\sup\left\{z\in P(Z(\cM)): \  z\Delta(p_n)\ge \frac{1}{2}z, \ zz_k=0, k=1,\ldots, n-1\right\},
$$
we get  $\Delta(z_np_n)\ge \frac{1}{2}z_n.$ So, we have constructed a sequence of mutually orthogonal central projections $z_1, \ldots, z_{n},\ldots$ in $\cM$ such that
$\Delta(z_np_n)\ge \frac{1}{2}z_n$ for all $n=1,2, \ldots.$
Since $\tau_\cM(\mathbf{1}-p_n)\rightarrow 0$ we have that
$\sum\limits_{n=1}^\infty z_n=\mathbf{1}.$

Since $\cM$ is of type II$_1$, for each $n\in \mathbb{N}$ we may
take a projection $f_n\le z_n p_n$ such that
$\Delta(f_n)=\frac{1}{2}z_n.$ Then $f_n\sim z_n-f_n$
in the reduced von Neumann algebra $z_n\cM.$ Let $\Phi_{f_n,
\widetilde{f_n}}$ be the mapping defined as in \eqref{pe}. We have
that $\Phi_{f_n, \widetilde{f_n}}(x)=\Phi_{p_n,
\widetilde{p_n}}(x)\widetilde{f_n}$ for all $x\in f_n \cM f_n,$
and  $\widetilde{f_n}\le \widetilde{p_n}.$ Therefore, since
$\Phi_{p_n, \widetilde{p_n}}$ is bijective, it follows that
$\Phi_{f_n, \widetilde{f_n}}$ maps $f_n \cM f_n$ onto
$\widetilde{f_n} \cN \widetilde{f_n}$ for all $n\in \mathbb{N}.$
By Lemma~\ref{ccc} for each $n\in \mathbb{N}$ there are real
$\ast$-isomorphism $\Psi_n$ from $S(z_n\cM)$ onto
$S(\Phi(z_n)\cN)$  and an invertible element $c_n\in
S(\Phi(z_n)\cN)$ such that
$$
\Phi(x)=c_n\Psi_n(x)c_n^{-1}
$$
for all $x\in S(z_n\cM).$
Setting
$c=\sum\limits_{n=1}^\infty c_n$ and
$$
\Psi(x)=\sum\limits_{n=1}^\infty \Psi_n(z_nx),\,\, x\in S(\cM),
$$
we obtain an invertible element $c\in S(\cN)$ and a
real $\ast$-isomorphism from $S(\cM)$ onto $S(\cN)$ such that
$$
\Phi(x)=c\Psi(x)c^{-1}
$$
for all $x\in S(\cM).$
The proof of Theorem is complete.
\end{proof}

\begin{proof}[Proof of Corollary~\ref{latticering}]
 Let $\cM$ and $\cN$  be von Neumann algebras of type
II$_1.$ Suppose that $\Theta: P(\cM) \to P(\cN)$ is a lattice isomorphism. By \cite[Part II, Theorem 4.2]{Neu60}  there exists  a ring  isomorphism $\Phi$  from $S(\cM)$ onto $S(\cN)$ such that
$\Theta(l(x))=l\left(\Phi(x)\right)$
for all $x\in S(\cM).$
Then by Theorem~\ref{latticeisomorphism} there exists a real
$\ast$-isomorphism $\Psi: \cM \to  \cN.$
The converse assertion is clear.
\end{proof}


\begin{thebibliography}{99}





%
%
%
\bibitem{Alb2} S. Albeverio,  Sh. Ayupov and K. Kudaybergenov, {\it
Structure of derivations on various algebras of measurable
operators for type I von Neumann algebras,} J.\ Funct.\ Anal.\
\textbf{256} (2009), 2917--2943.

\bibitem{AAKD11} S. Albeverio, S. Ayupov, K. Kudaybergenov, R. Djumamuratov, \textit{Automorphisms of central extensions
of type I von Neumann
algebras}, Studia Math.
\textbf{207} (2011),  1-17.


\bibitem{Berber}  S.K. Berberian, Baer $\ast$-rings. Die Grundlehren der mathematischen Wissenschaften, Band 195. Springer-Verlag, New York-Berlin, 1972.

\bibitem{Dales} H. Dales, Banach Algebras and Automatic Continuity, Clarendon Press, Oxford, 2000.



\bibitem{DP2} P. Dodds, B. de Pagter, {\it Normed K\"{o}the spaces: A non-commutative viewpoint,}
Indag. Math. {\bf 25} (2014) 206--249.



\bibitem{HS}
J. Huang, F. Sukochev, { \it Interpolation between $L_0(\cM,\tau)$ and $L_\infty(\cM,\tau)$}, Math. Z. {\bf 293} (2019), 1657--1672.



\bibitem{KRII}  R. Kadison and J. Ringrose, Fundamentals of the Theory of Operator Algebras, vol II, Academic Press, 1986.
%

\bibitem{KL} R. Kadison and  Z. Liu, {\it A note on derivations of Murray--von Neumann algebras,}
PNAS, {\bf 111}  (6)  (2014) 2087--2093.
%
%


\bibitem{KPR} N.J. Kalton, N.T. Peck, J.W. Roberts, An $F$-space sampler. London Mathematical Society Lecture Note Series, 89. Cambridge University Press, Cambridge, 1984.


\bibitem{Kus}{A.~G.~Kusraev,} \textit{Automorphisms and derivations in an extended complex
    $f$-algebra,} Sib. ~Math.~J. \textbf{47} (2006) 97--107.



\bibitem{Koliha} J. Koliha,
{\it Range  projections of idempotents in $C^\ast$-algebras,}
Demonstratio Mathematica, {\bf 34}, (2001) 91-103.


\bibitem{MC} M. A. Muratov, V. I. Chilin, Algebras of measurable and locally measurable operators, Institute of Mathematics
Ukrainian Academy of Sciences 2007.


\bibitem{MMori2020} M. Mori, {\it Lattice isomorphisms between projection lattices of von Neumann algebras,}
arXiv:2006.08959 (2020).


 \bibitem{Nel} E. Nelson, {\it  Notes on non-commutative integration,} J. Funct.
Anal. \textbf{15} (1974)  103--116.





\bibitem{Neu60} J. von Neumann, Continuous geometry, Foreword by Israel Halperin, Princeton
Mathematical Series, No. 25 Princeton University Press, Princeton, N.J. (1960).

\bibitem{Saito} K. Sait\^{o}, {\it On the algebra of measurable operators for a general $AW^*$-algebra.} II. Tohoku Math. J. {\bf 23} (1971), 525--534.


\bibitem{Sakai_book} S. Sakai, $C^*$-algebras and $W^*$-algebras. Reprint of the 1971 edition. Classics in Mathematics. Springer-Verlag, Berlin, 1998. xii+256 pp.

%

\bibitem{SHAC} T.A.Sarymsakov, Sh.A.Ayupov, Dj. Khadjiev, V.I.Chilin, Ordered algebras. V.I.Romanovsky Institute of Mathematics, Tashkent, 1983.


\bibitem{Segal}
I.E. Segal, {\it A non-commutative extension of abstract
integration}, Ann. Math. {\bf 57} (1953) 401--457.








\bibitem{Yos} K. Yosida, functional analysis. Springer Verlag, Berlin Heidelberg New York, 1980.



\end{thebibliography}
\end{document}